\documentclass[11pt, reqno]{amsart}

\usepackage[mathscr]{eucal}
\usepackage[all]{xy}
\usepackage{mathrsfs}
\usepackage{xypic}
\usepackage{amsfonts}
\usepackage{amsmath}
\usepackage{amsthm}
\usepackage{amssymb}
\usepackage[square, comma, sort&compress]{natbib}
\usepackage{eufrak}

\allowdisplaybreaks 
\newtheorem{theorem}{Theorem}
\newtheorem{lemma}[theorem]{Lemma}

\newtheorem{remark}[theorem]{Remark}

\newtheorem{corollary}[theorem]{Corollary}
\newtheorem{proposition}[theorem]{Proposition}

\theoremstyle{definition}
\newtheorem{definition}[theorem]{Definition}
\newtheorem{definition-lemma}[theorem]{Definition-Lemma}

\def\a{\alpha}
\def\b{\beta}
\def\g{\gamma}

\def\la{\longrightarrow}
\def\hp{\hphantom{x}}

\newcommand{\vi}{$\;\,{\sf {(i)}}\;$}
\newcommand{\vii}{$\;{\sf {(ii)}}\;$}
\newcommand{\viii}{${\sf {(iii)}}\;$}

\newcommand{\eps}{{\varepsilon}}
\newcommand{\E}{{\mathrm{E}}}
\newcommand{\M}{{\mathrm{M}}}
\newcommand{\ol}{\overline}
\newcommand{\ot}{{\otimes}}
\newcommand{\Q}{\mathbb{Q}}

\newcommand{\hh}{\widetilde{HH}}
\newcommand{\cc}{\widetilde{CC}}
\newcommand{\hc}{\widetilde{HC}}

\begin{document}

\title{Batalin-Vilkovisky coalgebra of string topology}

\author{Xiaojun Chen and Wee Liang Gan}

\date{}

\maketitle

\begin{abstract}
We show that the reduced Hochschild homology of a DG open 
Frobenius algebra has the natural structure of a Batalin-Vilkovisky
coalgebra, and the reduced cyclic homology has the natural structure
of a gravity coalgebra.
This gives an algebraic model for a Batalin-Vilkovisky
coalgebra structure on the reduced homology of the 
free loop space of a simply connected closed oriented manifold,
and a gravity coalgebra structure on the reduced equivariant homology.
\end{abstract}

\bigskip
\centerline{\sf Table of Contents}
\vskip -1mm

$\hspace{30mm}$ {\footnotesize \parbox[t]{115mm}{
\hp${}_{}$\!
\hp\!1.{ $\;\,$} {\tt Introduction.} \newline
\hp2.{ $\;\,$} {\tt Batalin-Vilkovisky algebras and gravity algebras.} \newline
\hp3.{ $\;\,$} {\tt The Batalin-Vilkovisky algebra.} \newline
\hp4.{ $\;\,$} {\tt The gravity algebra.} \newline
\hp5.{ $\;\,$} {\tt The Batalin-Vilkovisky coalgebra.}\newline
\hp6.{ $\;\,$} {\tt The gravity coalgebra.}\newline
}
}
\bigskip

\section{\bf Introduction}

Let $M$ be a simply connected closed oriented $m$-manifold and $LM$ its free loop space. 
Felix and Thomas \cite{FT07} gave a construction of the
Batalin-Vilkovisky algebra structure on the homology of $LM$
in terms of Hochschild homology of a Poincar\'e duality model
of $M$. The aim of this paper is to show
that the reduced Hochschild homology, which gives the
homology of $LM$ relative to constant loops,
has the structure of a Batalin-Vilkovisky coalgebra. As a consequence it is also shown that the reduced cyclic homology of the Poincar\'e duality model,
which models the equivariant homology of $LM$ relative to the constant loops, has the structure of a gravity algebra and coalgebra.

\subsection{}
Throughout this paper, we shall work over the field of rational numbers. By $C_*(-)$ and $C^*(-)$, we mean the complex of
singular chains and the complex of singular cochains. 
We shall grade $C^*(-)$ negatively.
Applying a recent theorem of Lambrechts and Stanley \cite[Theorem 1.1]{LS}
to the Sullivan minimal model of $M$, it follows that there is a 
commutative differential graded (DG) algebra $A$ such that:
\begin{itemize}
\item $A$ is connected, finite dimensional, and
quasi-isomorphic to the DG algebra $C^*(M)$;
\item There is an $A$-bimodule isomorphism $A\to A^\vee$ of 
degree $m$ commuting with the differential and inducing the 
Poincar\'e duality isomorphism $H^*(M)\to H_{m+*}(M)$ on homology.
\end{itemize} 
Following Felix and Thomas \cite{FT07}, we call $A$ 
a \emph{Poincar\'e duality model} for $M$. 

Let $C=A^\vee$, the dual space of $A$. Since $A$ is a commutative DG algebra,
$C$ is a cocommutative DG coalgebra. 
The linear isomorphism $A\stackrel{\cong}{\la}C[m]$ induces 
the structure of a commutative
DG algebra on $C$ whose product is of degree $-m$. Moreover,
the coproduct 
$$\Delta: C\la C\ot C, \quad x\mapsto x'\ot x''$$
is a map of  $C$-bimodules. Thus, $C$ forms a DG open 
Frobenius algebra in the following sense, which models the chain complex of $M$:

\begin{definition}[DG open Frobenius algebra]
Let $C$ be a chain complex over a field $k$. 
A \emph{DG open Frobenius algebra} of degree $m$ 
on $C$ is a triple $(C,\cdot, \Delta)$ such that 
$(C,\cdot)$ is a DG commutative algebra whose product is of 
degree $-m$; $(C,\Delta)$ is a DG cocommutative coalgebra;
and 
\begin{equation}  \label{frob}
(x\cdot y)'\ot (x\cdot y)'' = (x\cdot y')\ot y'' 
= (-1)^{m|x'|}x'\ot (x'' \cdot y) ,  
\quad\mbox{ for any }  x,y\in C. 
\end{equation}
\end{definition}

\subsection{}
From now on, we shall denote by $C$ a 
DG open Frobenius algebra of degree $m$
with differential $d$, counit $\eps$, and a coaugmentation $\mathbb Q\hookrightarrow C$.
By the Hochschild homology $HH_*(C)$
and cyclic homology $HC_*(C)$ of $C$, we mean the 
Hochschild homology and cyclic homology of 
the underlying DG \emph{coalgebra} structure of $C$. 
We recall their definitions:

\begin{definition}[Hochschild homology of coalgebra]
The \emph{Hochschild homology} $HH_*(C)$ of $C$ is 
the homology of the normalized cocyclic cobar complex 
$(CC_*(C),b)$, where
$$ CC_*(C) =  \prod_{n=0}^\infty C\otimes (\Sigma \ol{C})^{\ot n} ,$$
and 
\begin{eqnarray} \label{differentialb}
&& b(
a_0[a_1|\cdots|a_n])\label{twisteddiff}   \\ 
&:=&da_0[a_1|\cdots|a_n]
+\sum_{i=1}^n
(-1)^{|a_0|+|[a_1|\cdots|a_{i-1}]|}a_0 [a_1|\cdots|da_i|\cdots|a_n]\nonumber\\
&&+\sum_{i=1}^n(-1)^{|a_0|+|[a_1|\cdots|a_{i-1}|a_i']|}a_0
[a_1|\cdots|a_i'|a_i''|\cdots|a_n]\nonumber\\
&&+(-1)^{|a_0'|}a_0'
\Big([a_0''|a_1|\cdots|a_n]-(-1)^{(|a_0''|-1)|[a_1|\cdots|a_n]|}[a_1|\cdots|a_n|a_0'']\Big).\nonumber
\end{eqnarray}
\end{definition}

Here, $\ol{C}:=C/\mathbb Q\simeq \ker\{\eps:C\to\Q\}$, 
$\Sigma$ is the desuspension functor (shifting the degrees of
$\ol{C}$ down by one), and we write the elements of 
$C\ot (\Sigma\ol{C})^{\ot n}$ in the form $a_0[a_1|\cdots|a_n]$.
One easily checks that $b^2=0$.
Connes' cyclic operator on the normalized cocyclic cobar complex is given by
$$\begin{array}{cccl}
B: & CC_*(C)&\longrightarrow& CC_{*+1}(C)\\
&a_0 [a_1|\cdots|a_n]&\longmapsto&\displaystyle\sum_{i=1}^n
(-1)^{|[a_i|\cdots|a_n]||[a_1|\cdots|a_{i-1}]|} \varepsilon(a_0)a_i
[a_{i+1}|\cdots|a_n|a_1|\cdots|a_{i-1}].
\end{array}$$
One has $B^2=0$ and $bB+Bb=0$.
\begin{definition}[Cyclic homology of coalgebra]\label{cycliccomplex}
The \emph{cyclic homology} $HC_*(C)$ of $C$ is the
homology of the chain complex $CC_*(C)[u]$, where $u$ is a 
parameter of degree $2$, with differential $b+u^{-1}B$ defined by:
$$(b+u^{-1}B)(\a\otimes u^n)=\left\{\begin{array}{lcc}
b\a\otimes u^n+B\a\otimes u^{n-1},&\mbox{if}&n>0,\\
b\a,&\mbox{if}&n=0,\end{array}\right.$$ 
for $\a\in CC_*(C)$.
\end{definition}

As in the algebra case, one has Connes' exact sequence (cf. \cite[Theorem 8.3]{Chen}):
\begin{equation}\label{Connes_exact_seq}
\xymatrix{\cdots\ar[r]&HH_*(C)\ar[r]^\E& HC_*(C)\ar[r]& 
HC_{*-2}(C)\ar[r]^\M&HH_{*-1}(C)\ar[r]&\cdots.}
\end{equation}
We now recall the Batalin-Vilkovisky algebra structure on the
Hochschild homology of a DG open Frobenius algebra. First, let us recall the following:

\begin{definition}[Batalin-Vilkovisky algebra]\label{equiv_def_BV}
A \emph{Batalin-Vilkovisky algebra} 
is a graded commutative algebra $(V,\bullet)$ together with a
linear map $\Delta: V_* \to V_{*+1}$ such that 
$\Delta\circ\Delta=0$, and for all $a,b,c\in V$,
\begin{eqnarray} 
\Delta(a\bullet b\bullet c)&=&\Delta(a\bullet b)\bullet c+(-1)^{|a|}a\bullet\Delta(b\bullet c)
+(-1)^{(|a|-1)|b|}b\bullet\Delta(a\bullet c)     \nonumber \\
&&-(\Delta a)\bullet b\bullet c-(-1)^{|a|}a\bullet (\Delta b)\bullet
c-(-1)^{|a|+|b|}a\bullet b\bullet (\Delta c).     \label{secondorder}
\end{eqnarray}
\end{definition}

Now for a DG open Frobenius algebra $C$, define a product
$$\bullet: CC_*(C)\ot CC_*(C)\longrightarrow CC_*(C)$$
by
\begin{equation}
a_0 [a_1|\cdots|a_n]\bullet b_0  [b_1|\cdots|b_r]:=
(-1)^{|b_0||[a_1|\cdots|a_n]|}a_0b_0 [a_1|\cdots
|a_n|b_1|\cdots|b_r] .
\end{equation}

\begin{theorem}[Tradler]\label{theoremBV}
The Hochschild homology $HH_*(C)[m]$ is a 
Batalin-Vilkovisky algebra with differential $B$ 
and product $\bullet$.
\end{theorem}

Using the maps $\E$ and $\M$ in Connes' exact sequence
(\ref{Connes_exact_seq}),
define, for each integer $n\ge 2$,
\begin{eqnarray*}
c_n: HC_{*}(C)[m-2]^{\otimes n}&\la& HC_{*}(C)[m-2]\\
\a_1\otimes\cdots\otimes\a_n&\longmapsto&
(-1)^\epsilon \, \E(\M(\a_1)\bullet\cdots\bullet \M(\a_n)),
\end{eqnarray*}
where $\epsilon = (n-1)|\a_1|+(n-2)|\a_2|+\cdots+|\a_{n-1}|$.
It follows from a general result (see Proposition \ref{gravityalg}) that:
\begin{corollary}\label{corollarygravity}
The cyclic homology $(HC_{*}(C)[m-2],\{c_n\})$ is a gravity algebra (in the sense of the following definition).
\end{corollary}

\begin{definition}[Gravity algebra]\label{def_of_Lieinfty}
A \emph{gravity algebra} is a 
graded vector space $V$ with
a sequence of graded skew-symmetric operators
$$\{x_1,\cdots,x_k\}:V^{\otimes k}\longrightarrow V,\quad k=2,3,\cdots$$ of degree
$2-k$, such that they satisfy the following generalized Jacobi
identities: 
\begin{eqnarray}&&\sum_{1\le i<j\le
k}(-1)^{\epsilon(i,j)}\{\{x_i,x_j\},x_1,\cdots,\hat x_i,\cdots,\hat
x_j,\cdots,
x_k, y_1,\cdots,y_l\}\label{def_of_brackets}\\
&=&\left\{\begin{array}{ll}\{\{x_1,\cdots,x_k\},y_1,\cdots,y_l\},&l>0,\\
0,&l=0.\end{array}\right.\nonumber\end{eqnarray} where
$\epsilon(i,j)=(|x_1|+\cdots+|x_{i-1}|)|x_i|+(|x_1|+\cdots+|x_{j-1}|)|x_j|+|x_i||x_j|$.
\end{definition}

\subsection{}
In this paper, by the \emph{reduced Hochschild homology}  $\hh_*(C)$
of $C$, we mean the homology of
$$\cc_*(C) := CC_*(C)/C = \prod_{n=1}^\infty C\otimes (\Sigma \ol{C})^{\ot n} .$$ 
By the \emph{reduced cyclic homology} $\hc_*(C)$ of $C$, we mean the homology
of $$\cc_*(C)[u] = CC_*(C)[u]/C[u].$$
As above, we have $\E:\hh_*(C)\la \hc_*(C)$ and $\M:\hc_*(C)\la\hh_{*+1}(C)$.

Define a coproduct
$$
\vee: \cc_*(C)\longrightarrow \cc_*(C) \ot \cc_*(C)$$
by
\begin{equation}
\vee( a_0 [a_1|\cdots|a_n]):= \sum_{i=2}^{n-1} (-1)^{\epsilon(i)}\,
(a_0a_i)' [a_1|\cdots|a_{i-1}] 
\otimes (a_0a_i)'' [a_{i+1}|\cdots|a_n]
\end{equation}
where $\epsilon(i) = |a_0|+(1+|a_i|+|(a_0a_i)''|)
|[a_1|...|a_{i-1}]|$.

Our main result is the following.
\begin{theorem}\label{theoremcoBV}
The reduced Hochschild homology $\hh_{*}(C)[1-m]$ is a 
Batalin-Vilkovisky coalgebra with differential $B$ 
and coproduct $\vee$.
\end{theorem}
Similarly to above, define
$s_n: \hc_{*}(C)[2-m]\la \hc_{*}(C)[2-m]^{\otimes n}$ by
$$s_n(\a):=(\E\ot\cdots\ot \E)\circ(\vee\ot id^{\ot n-2})\circ\cdots\circ (\vee\ot id)\circ
\vee\circ \M(\a),$$
for any $\a\in \hc_{*}(C)[2-m]$. We have:
\begin{corollary}\label{corollarycogravity}
The reduced cyclic homology $(\hc_{*}(C)[2-m],\{s_n\})$ is a gravity coalgebra. 
\end{corollary}

In the above two statements, the Batalin-Vilkovisky coalgebra and gravity coalgebra are defined as a dual version
of the corresponding algebras (see Definitions~\ref{coBVdef} and \ref{def_of_gravity_coalg}). Note that the Batalin-Vilkovisky
algebra and the gravity algebra structures
of $HH_*(C)$ and $HC_*(C)$ descend to $\widetilde{HH}_*(C)$ and $\widetilde{HC}_*(C)$ respectively. Thus, we
obtain both Batalin-Vilkovisky algebra and coalgebra 
structures on $\hh_*(C)$, and gravity algebra and coalgebra structures on $\hc_*(C)$.

\subsection{}
Let $A$ be a Poincar\'e duality model for $M$ and $C=A^\vee$.
Let $LM$ be the free loop space of $M$.
From Jones \cite{Jo87}, one has isomorphisms 
$$H_*(LM,M) \cong \hh_*(C) , \quad H^{S^1}_*(LM,M) \cong \hc_*(C).$$ 
Following Chas and Sullivan \cite{CS02}, 
we call $H_*(LM,M)$ the reduced homology of the
free loop space, and $H^{S^1}_*(LM,M)$ the reduced equivariant homology of
the free loop space.
As a consequence, 
the choice of a Poincar\'e duality model for $M$ gives the 
reduced homology of the free loop space the structure of a
Batalin-Vilkovisky coalgebra, and the reduced equivariant homology
of the free loop space the structure of a gravity coalgebra.
In string topology, the loop product $\bullet$ 
was first introduced by Chas and Sullivan in \cite{CS99}; see also \cite{CJ}.
The coproduct $\vee$ was introduced by Sullivan in \cite{Su03}. 
The operators $c_n$ and the operators $s_n$ were first 
introduced by Chas and Sullivan in \cite{CS02} and discussed further in \cite{Su03}; see also \cite{Westerland}.

Batalin-Vilkovisky algebras and gravity algebras were studied by Getzler (\cite{Ge94,Ge1,Ge2}) in his 
works on topological conformal field theories (TCFT).
He showed that a (genus zero) TCFT (respectively, an equivariant TCFT) with one output is the same as a Batalin-Vilkovisky algebra (respectively, a gravity algebra). If we consider multiple
inputs and outputs, we then obtain both Batalin-Vilkovisky algebra and coalgebra (respectively, gravity algebra and coalgebra). Our construction gives an algebraic proof that
string topology is a part of a (genus zero) TCFT. 
We expect that the constructions above can be 
generalized to homotopy versions of DG open
Frobenius algebras.

\begin{remark} 
We emphasize that Sullivan's coproduct $\vee$ is not the same as the
loop coproduct introduced by Cohen and Godin in \cite{CG}; see also \cite{Godin}. 
\end{remark}
 
\begin{remark}
Theorem \ref{theoremBV} is not new;  it is well known that
the Hochschild cohomology of a Frobenius algebra 
has the structure of a Batalin-Vilkovisky algebra; see, for example,
\cite{Menichi} and \cite{Tradler}.
However, notice that the formulas we give above are really explicit 
and simple.
As far as we are aware, Theorem \ref{theoremcoBV} is new. 
\end{remark}

\subsection{}
This paper is organized as follows. We recall the definitions
of Batalin-Vilkovisky algebras and gravity algebras in 
Section 2 and the proof of Theorem \ref{theoremBV}
in Section 3. We give the proof of Corollary \ref{corollarygravity}
in Section 4, the proof of Theorem \ref{theoremcoBV} in 
Section 5, and the proof of Corollary \ref{corollarycogravity}
in Section 6. We 
shall adopt Koszul's rule for signs, that is, whenever we switch two elements
$a\otimes b\mapsto b\otimes a$, we put $(-1)^{|a||b|}$ in front of $b\otimes a$
and write $\pm b\ot a$.

\subsection*{Acknowledgments}
We thank F. Eshmatov for many discussions.
W.L.G. was partially supported by the NSF grant DMS-0726154.

\section{\bf Batalin-Vilkovisky algebras and gravity algebras}

\subsection{}We first recall some properties of Batalin-Vilkovisky algebras and gravity algebras.

\begin{lemma}Let $(V,\bullet,\Delta)$ be a Batalin-Vilkovisky
algebra. Define 
$$\{\, ,\, \}: V\ot V\la V $$
by
$$\{a,b\}:=(-1)^{|a|}\Delta(a\bullet b)-
(-1)^{|a|}(\Delta a)\bullet b- a\bullet(\Delta b).$$ 
Then $(V[-1], \{\, ,\, \}, \Delta)$ is a DG Lie algebra.
\end{lemma}

\begin{proof}See \cite[Proposition 1.2]{Ge94}.\end{proof}

More generally, one has the following result 
proved by Getzler (see \cite[Theorem 4.5]{Ge1}
and \cite[\S 3.4]{Ge2}).

\begin{theorem}\label{BVtoLieinfty}
Let $(V,\bullet,\Delta)$ be a Batalin-Vilkovisky algebra.
Define, for $k=2, 3,...$,
$$ \{\, ,\cdots,\, \}: V^{\ot k}\la V $$ by
$$\{a_1,\cdots, a_k\} := 
(-1)^\epsilon \big(
\Delta(a_1a_2\cdots a_k)-\sum_{i=1}^k(-1)^{|a_1|+\cdots+|a_{i-1}|}
a_1\cdots(\Delta a_i)\cdots a_k \big) $$
where $\epsilon=(k-1)|a_1|+(k-2)|a_2|+\cdots+|a_{k-1}|$.
Then $V[-1]$ is a DG gravity algebra 
with differential $\Delta$ and brackets $\{a_1,\cdots,a_k\}$.
\end{theorem}

A DG gravity algebra is a gravity algebra with a differential 
commuting with all the brackets.
Thus, for a Batalin-Vilkovisky algebra
$(V,\bullet,\Delta)$, its homology $H(V,\Delta)[-1]$ has a
gravity algebra structure. 
Note that taking $k=3$ and $l=0$ in (\ref{def_of_brackets})
gives the Jacobi identity. 
Hence, a gravity algebra has a Lie algebra structure.

\subsection{}
Analogously, we may introduce the notion of a Batalin-Vilkovisky coalgebra and the notion of a gravity coalgebra.

\begin{definition}[Batalin-Vilkovisky coalgebra]\label{coBVdef}
A \emph{Batalin-Vilkovisky coalgebra} is a 
graded cocommutative coalgebra $(V,\vee)$ together with a
linear map $\Delta: V_*\to V_{*+1}$ such that $\Delta\circ\Delta
=0$, and 
\begin{multline*}
(\Delta\otimes id^{\otimes 2}+id\otimes\Delta\otimes id+id^{\otimes 2}\otimes \Delta)\circ (\vee\otimes id)\circ \vee(a)  
=(\tau^2+\tau+id)\circ
(\vee\circ\Delta\otimes id)\circ\vee(a) \\
+ (\vee\ot id)\circ \vee \circ \Delta(a),
\end{multline*}
for all $a\in V$, where $\tau$ is the cyclic permutation $\tau:a\otimes b\otimes c\mapsto c\otimes a\otimes b$.
\end{definition}

Similarly to the Batalin-Vilkovisky algebra case, the chain complex $(V,\Delta)$ is a DG gravity coalgebra:

\begin{definition}[Gravity coalgebra]\label{def_of_gravity_coalg}
A \emph{gravity coalgebra} is a 
graded vector space $V$ with a sequence of 
graded skew-symmetric operators
$$m_k:V\longrightarrow V^{\otimes k},\quad k=2, 3, 4,\cdots$$
of degree $2-k$, such that
\begin{equation}\label{coLieinftyidentity}S_{2,k-2}\circ (m_2\otimes id^{\otimes k-2}) \circ m_{k-1+l}=(m_k\otimes id^{\otimes l})\circ m_{l+1}:V\longrightarrow V^{k+l},
\end{equation}
where the range of the mapping $(m_2\otimes id^{\otimes k-2}) \circ m_{k-1+l}:V\to V^{k+l}$ is identified with 
$V^{\otimes 2}\otimes V^{\otimes k-2}\otimes V^{\otimes l}$ and $S_{2,k-2}$ is the shuffle product $V^{\otimes 2}\otimes V^{\otimes k-2}\to V^{\otimes k}$, and if $l=0$, we set $m_1=0$.
\end{definition}

\begin{theorem}Let $(V,\vee,\Delta)$ be a Batalin-Vilkovisky coalgebra. For any $x\in V$, let
$$\vee_k(x):=(\vee\otimes id^{\otimes k-2})\circ\cdots\circ 
(\vee\ot id)\circ\vee(x)
= \sum x_1\ot x_2\ot \cdots \ot x_k,$$
and let
$$s_k(x):=  \sum (-1)^{(k-1)|x_1|+(k-2)|x_2|+\cdots+|x_{k-1}|}\Big(
\vee_k(\Delta x) -
\sum_{i=0}^{k-1} \big(
id^{\otimes i}\otimes \Delta\otimes id^{\otimes k-i-1}\big)\circ \vee_k(x) \Big) ,$$ 
for $k=2,3,\cdots.$
Then $V[1]$ is a DG gravity coalgebra with differential
$\Delta$ and cobrackets $\{s_n\}$.
In particular, $(V[1],s_2,\Delta)$ is a DG Lie coalgebra.
\end{theorem}

The proof of the theorem is completely dual to that of Theorem \ref{BVtoLieinfty}.

\section{\bf The Batalin-Vilkovisky algebra}

\subsection{} 
In this section, we recall the proof of Theorem \ref{theoremBV} from \cite{Chen}. 

\begin{lemma} \label{dgalg}
The chain complex $(CC_*(C)[m], b)$ is a DG algebra
with product $\bullet$.
\end{lemma}
\begin{proof}
The proof is by direct verification, see \cite[Lemma 4.1]{Chen}.
\end{proof}

The product $\bullet$ on $CC_*(C)[m]$ is not commutative, but homotopy
commutative:

\begin{lemma}\label{htpyofcomm}
Define a bilinear operator
$$*:CC_*(C)\otimes CC_*(C)\longrightarrow CC_*(C)$$
as follows: for $\a=a_0 [a_1|\cdots|a_n],\b=b_0 [b_1|\cdots|b_r]\in
CC_*(C)$,
\begin{equation}\label{starop}
\a*\b :=\sum_{i=1}^n
(-1)^{|b_0|+(|\b|-1)|[a_{i+1}|\cdots|a_n]|}
\varepsilon(a_i b_0)a_0
[a_1|\cdots|a_{i-1}|b_1|\cdots|b_r|a_{i+1}|\cdots|a_n].
\end{equation}
 Then
\begin{equation}\label{htpy}
b(\a*\b)=b\a*\b+(-1)^{|\a|+1}\a*b\b+(-1)^{|\a|}(\a\bullet\b-(-1)^{|\a||\b|}\b\bullet\a).\end{equation}
\end{lemma}
\begin{proof}
The proof is by direct verification, see \cite[Lemma 5.1]{Chen}.
\end{proof}

It follows from Lemma \ref{dgalg} and Lemma \ref{htpyofcomm}
that $(HH_*(C)[m],\bullet)$ is a graded commutative algebra.

\subsection{}
Define the binary operator
$$\{\,,\,\}: CC_*(C)\otimes CC_*(C)\longrightarrow CC_*(C)$$
to be the commutator of $*$ above, namely
$$\{\a,\b\}:=\a*\b-(-1)^{(|\a|+1)(|\b|+1)}\b*\a,$$
for $\a,\b\in CC_*(C)$.

\begin{lemma}
The chain complex $(CC_*(C)[m-1],b)$ is a 
 DG Lie algebra with the Lie bracket $\{\,,\,\}$.
\end{lemma}
\begin{proof}
The proof is by direct verification, see \cite[Lemma 5.4]{Chen}
and \cite[Corollary 5.5]{Chen}.
\end{proof}

In particular $HH_{*}(C)[m-1]$ is a graded
Lie algebra. Moreover, $\bullet$ and $\{\,,\,\}$ are compatible in the
following sense, which makes $HH_{*}(C)[m]$ into a Gerstenhaber algebra \cite{Ger63}:

\begin{definition}[Gerstenhaber algebra]\label{defofgerst}
Let $V$ be a graded vector space. A \emph{Gerstenhaber
algebra} on $V$ is a triple $(V,\cdot,\{\,,\,\})$ such that

\vi $(V,\cdot)$ is a graded commutative algebra; 

\vii $(V,\{\,,\,\})$ is a graded Lie algebra
whose Lie bracket is of degree $1$; 

\viii for any $\a,\b,\g\in V$, one has:
\begin{equation}\label{gerstiii}
\{\a\bullet\b,\gamma\}=\a\bullet\{\b,\gamma\}+(-1)^{|\b|(|\gamma|+1)}
\{\a,\gamma\}\bullet\b.\end{equation}
\end{definition}

\begin{theorem}\label{thmofgerst}
The Hochschild homology $HH_{*}(C)[m]$
is a Gerstenhaber algebra, with product $\bullet$ and bracket $\{\,,\,\}$.
\end{theorem}
\begin{proof}
From above, $HH_{*}(C)[m]$ is both a
graded commutative algebra and a degree one graded Lie algebra.
Equation (\ref{gerstiii}) is immediate from the following
Lemma \ref{lemma45}.
\end{proof}

\begin{lemma}\label{lemma45}
For any $\a=a_0[a_1|...|a_n],
\b=b_0[b_1|...|b_r],
\g=c_0[c_1|...|c_l]\in CC_*(C)$, one has:

\vi  $(\a\bullet\b)*\gamma=\a\bullet(\b*\gamma)+(-1)^{|\b|(|\gamma|+1)}(\a*\gamma)
\bullet\b;$

\vii $\gamma*(\a\bullet\b)-(\gamma*\a)\bullet\b-(-1)^{|\a|(|\gamma|+1)}\a\bullet(\gamma*\b)
=(b\circ h-h\circ b)(\a\otimes\b\otimes\gamma)$, where
\begin{align*}
h(\a & \otimes  \b\otimes\gamma) \\ 
:= & \sum_{i<j}(-1)^{\epsilon}
\varepsilon(c_ia_0)\varepsilon(c_jb_0)
c_0 [c_1|\cdots|c_{i-1}|a_1|\cdots|a_n|c_{i+1}|\cdots
|c_{j-1}|b_1|\cdots|b_r|c_{j+1}|\cdots|c_l], \\
\epsilon = &(|\a|-1)|[c_{i+1}|\cdots|c_n]|+(|\b|-1)|[c_{j+1}|\cdots|c_n]|.
\end{align*}
\end{lemma}

\begin{proof}
The proof is by direct verification, see \cite[Lemma 5.8]{Chen}.
\end{proof}

\subsection{}
Theorem \ref{theoremBV} follows from \cite[Proposition 1.2]{Ge94},
Theorem \ref{thmofgerst}, and the following:

\begin{lemma}\label{bvgersten}
For any $\a,\b\in HH_*(C)[m]$, one has
\begin{equation}
\{\a,\b\}=(-1)^{|\a|}B(\a\bullet\b)-(-1)^{|\a|}B(\a)\bullet
b-\a\bullet B(\b).
\end{equation}
More precisely, for
$\a=x [a_1|\cdots|a_n],\b=y [b_1|\cdots|b_r]\in
CC_*(C)$, define
$$
\phi(\a,\b):=\sum_{i<j}\pm\varepsilon(x)\varepsilon(a_jy)
a_i [a_{i+1}|\cdots|a_{j-1}|b_1|\cdots|b_r|a_{j+1}|\cdots|a_n|a_1|\cdots|a_{i-1}]
$$
and
$$\psi(\a,\b):=\sum_{k<l}\pm\varepsilon(y)\varepsilon(b_lx) b_k 
[b_{k+1}|\cdots|b_{l-1}|a_1|\cdots|a_n|b_{l+1}|\cdots|b_r|b_1|\cdots|b_{k-1}],
$$
and let $h:=\phi+\psi$. Then
\begin{equation}\label{bvhtpy}
(b\circ h+h\circ b)(\a\otimes
\b)=\{\a,\b\}-(-1)^{|\a|}B(\a\bullet\b)-(-1)^{(|\b|+1)(|a|+1)}\b\bullet
B(\a) +\a\bullet B(\b).
\end{equation}
\end{lemma}
\begin{proof}
The proof is by a direct verification, see \cite[Lemma 7.3]{Chen}.
\end{proof}

\section{\bf The gravity algebra}

\subsection{}
We define the complex $(CC_*(C)[u,u^{-1}], b+u^{-1}B)$ 
by 
$$ (b+u^{-1}B)(\a\ot u^n)= b\a\ot u^n + B\a \ot u^{n-1},
\quad\mbox{ for all } n. $$
The complex $(CC_*(C)[u], b+u^{-1}B)$ in 
Definition \ref{cycliccomplex} is the quotient 
of $(CC_*(C)[u,u^{-1}], b+u^{-1}B)$ 
by its subcomplex $CC_*(C)[u^{-1}]u^{-1}$.
The short exact sequence
$$ \xymatrix{
0 \ar[r] & CC_*(C) \ar[r] & 
CC_*(C)[u] \ar[r]^{u^{-1}} & CC_*(C)[u] \ar[r] & 0
}
$$
induces the long exact sequence (\ref{Connes_exact_seq}).
By diagram chasing, one can see that 
$$\M\circ \E=B:HH_*(C)\la HH_{*+1}(C).$$

\subsection{}
Corollary \ref{corollarygravity} is immediate from
Theorem \ref{theoremBV} and the following general result
(see \cite[Theorem 8.5]{Chen}):

\begin{proposition} \label{gravityalg}
Let $(V,\bullet,\Delta)$ be a Batalin-Vilkovisky algebra, and 
$W$ be a graded vector space.
Let $\E: V_*\to W_*$ and $\M: W_*\to V_{*+1}$
be two maps such that $\E\circ\M=0$ and $\M\circ\E=\Delta$. Then 
$(W[-2],\{c_n\})$ is a gravity algebra, where 
$$c_n(\a_1\otimes\cdots\otimes \a_n):= (-1)^{(n-1)|\a_1|+(n-2)|\a_2|+\cdots+|\a_{n-1}|}
\E(\M(\a_1)\bullet\cdots\bullet \M(\a_n)).$$
\end{proposition}
\begin{proof}
It follows from (\ref{secondorder}), by induction on $n$,  that
 \begin{eqnarray}
 \Delta(x_1\bullet
x_2\bullet\cdots\bullet x_n)&=&\sum_{i<j}\pm \Delta(x_i\bullet x_j)\bullet
x_1\bullet\cdots\bullet\widehat{x_i}\bullet\cdot \bullet\widehat{x_j}
\bullet\cdots\bullet x_n\label{deriv}\\
&& + (n-2) \sum_i \pm x_1\bullet\cdots\bullet \Delta x_i\bullet\cdots\bullet x_n .   \nonumber
\end{eqnarray}
Now let $x_i=\M(\a_i)$, and apply $\E$ to both sides of the above
equality; we obtain:
\begin{eqnarray*}
&& \E\circ \Delta(\M(\a_1)\bullet
\M(\a_2)\bullet\cdots\bullet \M(\a_n))\\
&=&\sum_{i<j}\pm \E\circ\Big( \Delta(\M(\a_i)\bullet \M(\a_j))\bullet
\M(\a_1)\bullet\cdots\bullet\widehat{\M(\a_i)}\bullet\cdots\bullet\widehat{\M(\a_j)}
\bullet\cdots\bullet \M(\a_n)\Big)\\
&&+ (n-2)\sum_i \pm \E( \M(\a_1)\bullet\cdots\bullet \Delta\circ
\M(\a_i)\bullet\cdots\bullet \M(\a_n)) .
\end{eqnarray*} 
Since $\E\circ\Delta=\E\circ \M\circ \E=0$ and  $\Delta\circ \M=\M\circ \E\circ \M=0$, we have 
$$\sum_{1\le i< j\le n}\pm
c_{n-1} (c_2(\a_i\ot\a_j)\ot\a_1\ot\cdots\ot\widehat{\a_i}\ot\cdots\ot\widehat{\a_j}\ot
\cdots\ot\a_n)=0.$$ 

Similarly, by multiplying $y_1\bullet\cdots\bullet y_l$ on
both sides of (\ref{deriv}),  letting $y_j =\M(\b_j)$, and then applying $\E$ on
both sides, we obtain
\begin{eqnarray*}
&&\sum_{1\le i< j\le n}\pm c_{n+l-1}(
c_2(\a_i\ot \a_j)\ot\a_1\ot\cdots\ot\widehat{\a_i}\ot\cdots\ot\widehat{\a_j}\ot
\cdots\ot \a_n\ot\b_1\ot\cdots\ot\b_l) \\
&=&c_{l+1}(c_n(\a_1\ot\cdots\ot\a_n)\ot\b_1\ot\cdots\ot\b_l),
\end{eqnarray*}
for $l>0$. This proves the proposition.
\end{proof}

\section{\bf The Batalin-Vilkovisky coalgebra}

\subsection{} The proof of Theorem \ref{theoremcoBV} is similar to the proof
of Theorem \ref{theoremBV}. We begin with the following lemma.

\begin{lemma} \label{coalgebra}
The chain complex $(\cc_*(C)[1-m],b)$ is a DG coalgebra
with coproduct $\vee$.
\end{lemma}
\begin{proof}
It is clear that $\vee$ is coassociative. Therefore we only
need to check that $b$ is a derivation with respect to $\vee$.
Observe that the expressions $b\circ\vee(\a)$ and $\vee\circ b(\a)$ have two parts,
one contains those terms involving the differentials of the entries in $\a$ (which we call the {\it differential part}), the other contains those
terms involving the coproducts of the entries in $\a$ (which we call the {\it diagonal part}). It follows directly from the definition of $\vee$
that the differential parts of $b\circ\vee(\a)$ and $\vee\circ b(\a)$ are equal. For the diagonal parts, omitting the signs determined by Koszul sign rule from
our notation, we have
\begin{eqnarray}
&& b\circ\vee(a_0[a_1|...|a_n]) \label{dgc1} \\
&=& b\Big(\sum_{1<i<n} (a_0a_i)'[a_1|...|a_{i-1}]\Big)\ot
  (a_0a_i)''[a_{i+1}|...|a_n] \label{dgc2}\\
&\pm& \sum_{1<i<n} (a_0a_i)'[a_1|...|a_{i-1}]\ot
  b\Big( (a_0a_i)''[a_{i+1}|...|a_n] \Big) \label{dgc3}\\
&=&\sum_{1\le j<i<n}\pm (a_0a_i)'[a_1|...|a'_j|a''_j|...|a_{i-1}]
           \ot (a_0a_i)''[a_{i+1}|...|a_n] \label{dgc8}\\
&+&\sum_{1<i<n}\pm ((a_0a_i)')'[((a_0a_i)')''|a_1|...|a_{i-1}]
           \ot (a_0a_i)''[a_{i+1}|...|a_n] \label{dgc9}\\
&-&\sum_{1<i<n} \pm ((a_0a_i)')'[a_1|...|a_{i-1}|((a_0a_i)')'']
           \ot (a_0a_i)''[a_{i+1}|...|a_n] \label{dgc10}\\
&+&\sum_{1<i<j\le n}\pm (a_0a_i)'[a_1|...|a_{i-1}]
           \ot (a_0a_i)''[a_{i+1}|...|a'_j|a''_j|...|a_n] \label{dgc11}\\
&+&\sum_{1<i<n}\pm (a_0a_i)'[a_1|...|a_{i-1}]
           \ot ((a_0a_i)'')'[((a_0a_i)'')''|a_{i+1}|...|a_n] \label{dgc12}\\      
&-&\sum_{1<i<n} \pm(a_0a_i)'[a_1|...|a_{i-1}]
           \ot ((a_0a_i)'')'[a_{i+1}|...|a_n|((a_0a_i)'')''] , \label{dgc13}
\end{eqnarray}
while
\begin{eqnarray}
&& \vee\circ b(a_0[a_1|...|a_n]) \label{dgc21} \\
&=&\sum_{1\le j<i<n}\pm (a_0a_i)'[a_1|...|a'_j|a''_j|...|a_{i-1}]\ot
  (a_0a_i)''[a_{i+1}|...|a_n] \label{dgc26}\\
&+&\sum_{1<i<j\le n}\pm (a_0a_i)'[a_1|...|a_{i-1}]
           \ot (a_0a_i)''[a_{i+1}|...|a'_j|a''_j|...|a_n] \label{dgc27}\\
&+&\sum_{1<i<n}\pm (a_0a'_i)'[a_1|...|a_{i-1}]
           \ot (a_0a'_i)''[a''_i|a_{i+1}|...|a_n] \label{dgc28}\\
&+&\sum\pm (a_0a'_n)'[a_1|...|a_{n-1}]\ot (a_0a'_n)''[a''_n] 
{\phantom{\sum_i}} \label{dgc29}\\
&+&\sum \pm (a_0a''_1)'[a'_1]\ot (a_0a''_1)''[a_2|...|a_n]
{\phantom{\sum_i}} \label{dgc30}\\
&+& \sum_{1<i<n}\pm (a_0a''_i)'[a_1|...|a_{i-1}|a'_i]
           \ot (a_0a''_i)''[a_{i+1}|...|a_n] \label{dgc31} \\                     
&+&\sum\pm (a'_0a_1)'[a''_0]\ot (a'_0a_1)''[a_2|...|a_n] 
{\phantom{\sum_i}} \label{dgc32}\\
&+& \sum_{1<i<n}\pm (a'_0a_i)'[a''_0|a_1|...|a_{i-1}]
           \ot (a_0a_i)''[a_{i+1}|...|a_n] \label{dgc33}\\
&-&\sum_{1<i<n}\pm (a'_0a_i)'[a_1|...|a_{i-1}]
           \ot (a_0a_i)''[a_{i+1}|...|a_n|a''_0] \label{dgc34}\\
&-&\sum\pm (a'_0a_n)'[a_1|...|a_{n-1}]\ot (a'_0a_n)''[a''_0] {\phantom{\sum_i}} . \label{dgc35}
\end{eqnarray}
Keeping (\ref{frob}) in mind, we see that
(\ref{dgc8}) and (\ref{dgc26}) cancel, so do
(\ref{dgc9}) and (\ref{dgc33}),
(\ref{dgc10}) and (\ref{dgc31}),
(\ref{dgc11}) and (\ref{dgc27}),
(\ref{dgc12}) and (\ref{dgc28}),
(\ref{dgc13}) and (\ref{dgc34}). Hence, (\ref{dgc1}) = (\ref{dgc21}).
\end{proof}

\subsection{}
Define the permutations $\tau$ and $\sigma$ by
\begin{align*}
 \tau: \cc_*(C)\ot \cc_*(C) & \to \cc_*(C)\ot \cc_*(C)\\
  \a_1\ot\a_2 & \mapsto \pm\a_2\ot\a_1
\end{align*}
and
\begin{align*}
 \sigma: {\cc}_*(C) \otimes {\cc}_*(C)\otimes {\cc}_*(C) & \to
 {\cc}_*(C)\otimes {\cc}_*(C)\otimes {\cc}_*(C) \\
 \a_1\otimes\a_2\otimes\a_3 & \mapsto \pm \a_2\otimes\a_3\otimes\a_1.
\end{align*}
The following lemma says that $\vee$ is
cocommutative up to homotopy, 
and so $(\hh_*(C)[1-m], \vee)$ is a graded
cocommutative, coassociative coalgebra.

\begin{lemma} \label{cocommute}
Let $h:{\cc}_*(C)\longrightarrow {\cc}_*(C)\otimes{\cc}_*(C)$ be defined
by $$h(\a):=\sum_{i<j} \pm
a_0 [a_1|\cdots|a_{i-1}|a_{j+1}|\cdots|a_n] \otimes
 a_ia_j [a_{i+1}|\cdots|a_{j-1}],$$
for any $\a=a_0 [a_1|\cdots|a_n]\in {\cc}_*(C)$.
Then
\begin{equation}\label{cohtpy}
b\circ h(\a)+h\circ b(\a)=\tau\circ\vee(\a)-\vee(\a).
\end{equation}
\end{lemma}

\begin{proof}
It is easy to see that the differential parts of the left
hand side of (\ref{cohtpy}) cancel each other, so we only need to
consider the diagonal parts. In fact, the diagonal parts of $h(b\a)$
are equal to
\begin{eqnarray}
&&\sum
a_0'[a_{i+1}|\cdots|a_n]\otimes(a_0''a_i)[a_1|\cdots|a_{i-1}]\label{cty1}\\
&\pm&\sum a_0'[a_0''|a_1|\cdots|a_{i-1}|a_{j+1}|\cdots|a_n]\otimes
(a_ia_j)[a_{i+1}|\cdots|a_{j-1}]\label{bhalpha1}\\
&\mp&\sum
a_0'[a_1|\cdots|a_{i-1}]\otimes(a_ia_0'')[a_{i+1}|\cdots|a_n]\label{cty2}\\
&\mp&\sum a_0'[a_1|\cdots|a_{i-1}|a_{j+1}|\cdots|a_n|a_0'']\otimes
(a_ia_j)[a_{i+1}|\cdots|a_{j-1}]\label{bhalpha2}\\
&\pm&\sum
a_0[a_1|\cdots|a_k'|a_k''|\cdots|a_n]\otimes(a_ia_j)[a_{i+1}|\cdots|a_{j-1}]\label{bhalpha3}\\
&\pm&\sum
a_0[a_1|\cdots|a_{i-1}|a_{j+1}|\cdots
|a_n]\otimes(a_ia_j)[a_{i+1}|\cdots|a_k'|a_k''|\cdots
|a_{j-1}]\label{bhalpha11}\\
&\pm&\sum a_0[a_1|\cdots|a_{i-1}|a_{j+1}|\cdots|a_n]\otimes
(a_ia_j)' [(a_ia_j)''|a_{i+1}|\cdots|a_{j-1}]\label{bhalpha4}\\
&\mp&\sum a_0[a_1|\cdots|a_{i-1}|a_{j+1}|\cdots|a_n]\otimes
(a_ia_j)' [a_{i+1}|\cdots|a_{j-1}|(a_ia_j)'']\label{bhalpha5},
\end{eqnarray}
and
(\ref{bhalpha1})+(\ref{bhalpha2})+(\ref{bhalpha3})
+(\ref{bhalpha11})+(\ref{bhalpha4})+(\ref{bhalpha5})
is exactly $b(h\a)$, while the remaining terms
(\ref{cty1})+(\ref{cty2}) are exactly $\tau\circ\vee(\a)-\vee(\a)$.
The lemma is proved.
\end{proof}

\begin{lemma} \label{coalgs}
Let $h$ be as in Lemma \ref{cocommute}. Define
$S: {\cc}_*(C)\longrightarrow {\cc}_*(C)\otimes{\cc}_*(C)$ by
$$S(\a):=h(\a)-\tau\circ h(\a),\quad\mbox{for any }\a\in{\cc}_*(C).$$
Then the chain complex
$({\cc}_*(C)[2-m], b)$ is a  DG Lie coalgebra with the
cobracket $S$.
\end{lemma}

\begin{proof}
It follows from the definition that $S$ is skew-symmetric. 
Moreover,  $b$ commutes with $S$ by (\ref{cohtpy}).
Now, for any $\a=a_0[a_1|\cdots|a_n]$, 
\begin{eqnarray*}
&& (h\ot 1)h(\a) - (1\ot h)h(\a)  \\
&=& \sum_{k<l<i<j} \pm a_0[a_1|\cdots|a_{k-1}|a_{l+1}|\cdots|a_{i-1}|a_{j+1}|\cdots|a_n]  \ot a_ka_l[a_{k+1}|\cdots|a_{l-1}]\ot a_ia_j[a_{i+1}|\cdots|a_{j-1}] \\
&+& \sum_{i<j<k<l} \pm a_0[a_1|\cdots|a_{i-1}|a_{j+1}|\cdots|a_{k-1}|a_{l+1}|\cdots|a_n] \ot a_ka_l[a_{k+1}|\cdots|a_{l-1}]\ot a_ia_j[a_{i+1}|\cdots|a_{j-1}] \\
&=& (1\ot \tau) ((h\ot 1)h(\a) - (1\ot h)h(\a)) .
\end{eqnarray*}
It follows that
$$ (1+\sigma+\sigma^2) (S\ot 1)S = (1+\sigma+\sigma^2)
\Big( (h\ot 1)h - (1\ot h)h - (1\ot \tau) \big((h\ot 1)h - (1\ot h)h\big) \Big) = 0,
$$
so the co-Jacobi identity holds.
\end{proof}

It follows that $(\hh_*(C)[2-m],S)$ is a graded Lie coalgebra.
The Lie cobracket $S$ and the cocommutative coproduct $\vee$ are compatible in the following sense:

\begin{definition}[Gerstenhaber coalgebra]\label{cogersten}
Let $V$ be a graded vector space. A 
\emph{Gerstenhaber coalgebra} on $V$ is a 
triple $(V,\vee,S)$ which satisfies
the following:

\vi $(V,\vee)$ is a graded cocommutative coalgebra;

\vii $(V,S)$ is a graded Lie coalgebra whose Lie cobracket is
of degree 1;

\viii $S:V\to V\otimes V$ is a coderivation with respect to $\vee$, i.e. the following diagram commutes:
$$
\xymatrix{
V\ar[r]^\vee\ar[d]_{S}&V\otimes V\ar[d]^{(id\otimes\tau)\circ(S\otimes id)+id\otimes S}\\
V\otimes V\ar[r]^-{\vee\otimes id}&V\otimes V\otimes V
}$$
\end{definition}

\begin{theorem} \label{thmcogerst}
The reduced Hochschild homology
$(\hh_*(C)[1-m], \vee, S)$ is a Gerstenhaber coalgebra.
\end{theorem}

\begin{proof}
 From the definition of $h$, the following diagram commutes
$$
\xymatrix{
 {\cc}\ar[r]^\vee\ar[d]_{h}& {\cc}\otimes  {\cc}
\ar[d]^{(id\otimes\tau)\circ(h\otimes id)+id\otimes h}\\
 {\cc}\otimes {\cc}\ar[r]^-{\vee\otimes id}& {\cc}\otimes  {\cc}\otimes  {\cc}.
}$$
We next show that the following diagram commutes up to homotopy:
\begin{equation}\label{cogersten_htpy}
\xymatrix{
 {\cc}\ar[r]^\vee\ar[d]_{\tau\circ h}& {\cc}\otimes
 {\cc}\ar[d]^{(id\otimes\tau)\circ(\tau\circ h\otimes id)+id\otimes\tau\circ h}\\
 {\cc}\otimes  {\cc}\ar[r]^-{\vee\otimes
id}& {\cc}\otimes  {\cc}\otimes  {\cc},
}\end{equation} 
and therefore, from $S=h-\tau\circ h$, the following
diagram commutes:
$$ \xymatrix{
 {\hh}\ar[r]^\vee\ar[d]_{S}& {\hh}\otimes
 {\hh}
\ar[d]^{(id\otimes\tau)\circ(S\otimes id)+id\otimes S}\\
 {\hh}\otimes  {\hh}\ar[r]^-{\vee\otimes id}
& {\hh}\otimes  {\hh}\otimes  {\hh}. }$$

To this end, let $\psi: {\cc}\longrightarrow {\cc}\otimes {\cc}\otimes {\cc}$
be the map defined by
\begin{eqnarray*}\psi(\a)&:=&\sum_{i<j<k<l}\pm a_0[a_1|\cdots|a_{i-1}|a_{j+1}|\cdots|a_{k-1}|a_{l+1}|\cdots|a_n]\\
&&\quad\quad\quad\quad\quad\quad\otimes
a_ia_j [a_{i+1}|\cdots|a_{j-1}]\otimes a_ka_l[a_{k+1}|\cdots|a_{l-1}],\end{eqnarray*}
for any $\a=a_0[a_1|\cdots|a_n]$.
Let $\phi:=\sigma\circ\psi$.
Then
\begin{equation}\label{ht_1}(b\circ\phi+\phi\circ b)(\a)=\big((\vee\otimes id)\circ (\tau\circ h)-((id\otimes\tau)\circ(\tau\circ h\otimes id)
+id\otimes\tau\circ h)\circ\vee\big)(\a),\end{equation} for any $\a\in
 {\cc}$.
Indeed, one has 
 \begin{eqnarray}
 && \psi\circ b(\a) \nonumber \\
 &=& -b\circ\psi(\a) \nonumber \\
 &+&\sum_{i<j<k}\pm  (a_0a_i)'[a_{i+1}|\cdots|a_{j-1}|a_{k+1}|\cdots|a_n]\otimes
 (a_0a_i)''[a_1|\cdots|a_{i-1}]\otimes
 (a_ja_k)[a_{j+1}|\cdots|a_{k-1}]\quad
 \label{rem_1}\\
&+&\sum_{j<i<k}\pm a_0[a_1|\cdots|a_{j-1}|a_{k+1}|\cdots|a_n]\otimes(a_ja_ka_i)'[a_{j+1}|\cdots|a_{i-1}]\otimes(a_ja_ka_i)''[a_{i+1}|\cdots|a_{k-1}]\quad\quad 
 \label{rem_2}\\
 &+&\sum_{j<k<i}\pm (a_0a_i)'[a_1|\cdots|a_{j-1}|a_{k+1}|\cdots|a_{i-1}]\otimes (a_ja_k)[a_{j+1}|\cdots|a_{k-1}]\otimes(a_0a_i)''[a_{i+1}|\cdots|a_n].
 \label{rem_3} 
 \end{eqnarray}
 After applying $\sigma$, 
 (\ref{rem_1}) becomes $(id\otimes\tau\circ h)\circ\vee(\a)$,
 $(\ref{rem_2})$ becomes $(\vee\otimes id)\circ(\tau\circ h)(\a)$,
 and $(\ref{rem_3})$ becomes $(id\otimes\tau)\circ(\tau\circ h\otimes id)\circ\vee(\a)$. 
 This proves the identity (\ref{ht_1}), and hence
(\ref{cogersten_htpy}) is proved.
\end{proof}

\subsection{}
Theorem \ref{theoremcoBV} follows 
from the dual version of \cite[Proposition 1.2]{Ge94},
Theorem \ref{thmcogerst},
and the following lemma.

\begin{lemma}For any $\a=a_0 [a_1|\cdots|a_n]\in {\cc}_*(C)$, let 
\begin{eqnarray}\phi(\a)&:=&\sum_{i<j<k}\pm\varepsilon(a_0)a_i [a_{i+1}|\cdots|a_{j-1}|a_{k+1}|\cdots|a_n|a_1|\cdots|a_{i-1}]\otimes
a_ja_k [a_{j+1}|\cdots|a_{k-1}],\label{htpy_11}\\
\psi(\a)&:=&\sum_{j<k<i}\pm\varepsilon(a_0) a_ja_k[a_{j+1}|\cdots|a_{k-1}]\otimes a_i[a_{i+1}|\cdots|a_n|a_1|\cdots|a_{j-1}|a_{k+1}|\cdots|a_{i-1}],\label{htpy_22}
\end{eqnarray}
and let $\theta=\phi+\psi$. 
Then
$$b\circ \theta + \theta \circ b=\vee\circ B-B\circ\vee-S,$$
where $S$ is as defined in Lemma \ref{coalgs}.
\end{lemma}

\begin{proof}The proof is similar to that of Lemma \ref{bvgersten}. 
For any $\a=a_0[a_1|\cdots|a_n]$,
\begin{eqnarray}
\vee\circ B(\a)&=&\sum_{i>j}\pm\varepsilon(a_0)(a_ia_j)'[a_{i+1}|\cdots|a_n|a_1|\cdots|a_{j-1}]\otimes (a_ia_j)''[a_{j+1}|\cdots|a_{i-1}]\label{terms_11}\\
&&+\sum_{i<j}\pm\varepsilon(a_0)(a_ia_j)'[a_{i+1}|\cdots|a_{j-1}]\otimes (a_ia_j)''[a_{j+1}|\cdots|a_{n}|a_{1}|\cdots|a_{i-1}],\label{terms_12}\\
B\circ \vee(\a)&=&\sum_{i>k}\pm a_k[a_{k+1}|\cdots|a_{i-1}|a_1|\cdots|a_{k-1}]\otimes a_0a_i[a_{i+1}|\cdots|a_n]\label{terms_21}\\
&+&\sum_{i<k} \pm a_0a_i[a_1|\cdots|a_{i-1}]\otimes  a_k[a_{k+1}|\cdots|a_n|a_{i+1}|\cdots|a_{k-1}],\label{terms_22}\\
S(\a)&=&\sum_{i<j}\pm a_0[a_1|\cdots|a_{i-1}|a_{j+1}|\cdots|a_n]\otimes a_ia_j[a_{i+1}|\cdots|a_{j-1}]\label{terms_31}\\
&+&\sum_{i<j}\pm a_ia_j[a_{i+1}|\cdots|a_{j-1}]\otimes a_0[a_1|\cdots|a_{i-1}|a_{j+1}|\cdots|a_n].\label{terms_32}
\end{eqnarray}
It follows that
$$ \phi\circ b(\a) = -b\circ\phi(\a) + (\ref{terms_11})+(\ref{terms_21})-(\ref{terms_31}), $$
while
$$\psi\circ b(\a) = -b\circ\psi(\a)+(\ref{terms_12})+(\ref{terms_22})-(\ref{terms_32}).$$
This proves the lemma.
\end{proof}

\section{\bf The gravity coalgebra}

\subsection{}
Corollary \ref{corollarycogravity} is immediate from
Theorem \ref{theoremcoBV} and the following result.

\begin{proposition}
Let $(V,\vee,\Delta)$ be a Batalin-Vilkovisky coalgebra,
and $W$ be a graded vector space. Let $\E:V_*\to W_*$ and
$\M: W_*\to V_{*+1}$ be two maps such that $\E\circ \M=0$ 
and $\M\circ\E=\Delta$. 
Define $s_n:W\la W^{\otimes n}$  ($n\ge 2$)
by
$$s_n(\a):=(\E\otimes\cdots\otimes \E)\circ(\vee\otimes id^{\otimes n-2})\circ\cdots\circ\vee\circ \M(\a),$$
for any $\a\in W$. Then $(W[1],\{s_n\})$ is a gravity coalgebra.
\end{proposition}

\begin{proof}
The proof is analogous to that of Proposition \ref{gravityalg}.
By induction on $n$, we deduce from the identity in 
Definition \ref{coBVdef} that
\begin{equation}\label{coBVidentity}
\vee_n\circ\Delta( x)-
(n-2)(\sum_{i=1}^{n-1} id^{\otimes i}\otimes\Delta\otimes 
id^{\otimes n-i-1})\circ\vee_n(x)
=S_{2,n-2}\circ(\vee\circ\Delta\otimes id^{\otimes n-2})\circ\vee_{n-1}(x), \end{equation} for all $x\in V$,
where we set $\vee_n:=(\vee\otimes id^{\otimes n-2})\circ\cdots\circ\vee:V\to V^{\otimes n}$ as before. 

Let $x=\M(\a)$ where $\a\in W$. 
Applying $\E^{\otimes n}$ to both sides of (\ref{coBVidentity}),
we get
\begin{eqnarray*}
&& \E^{\otimes n}\circ\Big( \vee_n\circ \Delta(\M(\a))-
(n-2)\sum_{i=1}^{n-1} id^{\otimes i}\otimes 
\Delta\otimes id^{\otimes n-i-1})\circ\vee_n(\M(\a))\Big)\\
&=&
\E^{\otimes n}\circ S_{2,n-2}\circ(\vee\circ 
\Delta\otimes id^{\otimes n-2})\circ\vee_{n-1}(\M(\a)),\end{eqnarray*}
where the left hand side vanishes since 
$\Delta=\M\circ \E$ and $\E\circ \M=0$. Hence, we have
\begin{eqnarray*}
0&=& \E^{\otimes n}\circ S_{2,n-2}\circ(\vee\circ 
\Delta\otimes id^{\otimes n-2})\circ \vee_{n-1}(\M(\a))\\
&=&\E^{\otimes n}\circ S_{2,n-2}\circ(\vee\circ \M\circ \E\otimes id^{\otimes n-2})\circ \vee_{n-1}(\M(\a))\\
&=&S_{2,n-2}\circ(\E^{\otimes 2}\circ\vee\circ(\M\circ \E)\otimes \E^{\otimes n-2})\circ\vee_{n-1}(\M(\a))\\
&=&S_{2,n-2}\circ(s_2\otimes id^{\otimes n-2})\circ s_{n-1}(\a).\end{eqnarray*} This proves the identity (\ref{coLieinftyidentity}) in the definition of a gravity coalgebra for the case $l=0$.

Now let $l>0$. Let $x=\M(\a)$ where $\a\in W$
and suppose
 $$\vee_{l+1}(x)=x_1\otimes \cdots\otimes x_{l+1}.$$
Applying the identity (\ref{coBVidentity}) to the first component on both sides, by the same argument as above, we obtain:
$$S_{2,n-2}\circ (s_2\otimes id^{\otimes n-2}) \circ s_{n-1+l}(\a)
=(s_n\otimes id^{\otimes l})\circ s_{l+1}(\a).$$
This proves the identity (\ref{coLieinftyidentity}) for the case $l>0$. 
\end{proof}

\bibliographystyle{plain}

\smallskip

\footnotesize{

\noindent
{\bf X.C.}:
Department of Mathematics, University of Michigan,
Ann Arbor,  MI 48109, USA;\\
\hphantom{x}\quad\, {\tt xch@umich.edu}      

\smallskip

\noindent
{\bf W.L.G.}:
Department of Mathematics, University of California,
Riverside, CA 92521, USA;\\
\hphantom{x}\quad\, {\tt wlgan@math.ucr.edu}   }

\end{document}